%% file: main.tex
\PassOptionsToPackage{hypertexnames=false}{hyperref}
\documentclass[opre,nonblindrev]{informs4_revised}
\OneAndAHalfSpacedXI 

\usepackage{bbm}
\usepackage{fix-cm}
\newif\ifarxiv
\arxivtrue

\input{premeable}

\renewcommand{\Pr}{\mathbb{P}}
\providecommand{\E}{\mathbb{E}}
\providecommand{\R}{\mathbb{R}}

\providecommand{\cA}{\mathcal{A}}

\providecommand{\cA}{\mathcal{A}}
\providecommand{\cP}{\mathcal{P}}

\renewcommand{\d}{\mathrm{d}}

\newcommand{\TSP}{\textsf{TSP}}

\usepackage[normalem]{ulem}

\usepackage{natbib}
\bibpunct[, ]{(}{)}{,}{a}{}{,}
\def\bibfont{\small}

\definecolor{longhorn}{rgb}{0.8, 0.33, 0.0}

\TheoremsNumberedThrough     
\ECRepeatTheorems

\EquationsNumberedThrough   
\usepackage{amssymb}
\usepackage{graphicx}
\usepackage{pgfplots}
\usepackage{hyperref}
\hypersetup{
	colorlinks=true,
	linkcolor=blue,
	filecolor=blue,      
	urlcolor=red,
	citecolor=blue,
}

\usepackage{tikz}
\usetikzlibrary{patterns,arrows.meta,decorations.markings,calc}
\definecolor{tourblue}{rgb}{0,0.447,0.698}
\definecolor{delorange}{rgb}{0.835,0.369,0}
\usepackage{mathrsfs}
\usepackage{bbm}
\usepackage{multirow}
\usepackage{diagbox}
\usepackage{extpfeil}
\usepackage{tabularx}
\usepackage{color}
\usepackage{subcaption}
\usepackage{booktabs} 
\usepackage{array}

\providecommand{\Poi}{\operatorname{Poisson}}
\providecommand{\Exp}{\operatorname{Exp}}
\providecommand{\Unif}{\operatorname{Unif}}
\providecommand{\cC}{\mathcal{C}}
\providecommand{\cQ}{\mathcal{Q}}

\begin{document}
	
	\ABSTRACT{
		The Beardwood-Halton-Hammersley theorem characterizes the asymptotic length of the optimal Euclidean traveling salesman tour through independent and uniformly distributed random points $X_1,\ldots,X_n$ in the unit square. It states that there exists a universal constant $\beta$ such that the length of the shortest tour is asymptotic to $\beta \sqrt{n}$ almost surely. The best bounds established to date are $0.6277\leq \beta \leq 0.90367$. In this paper, we improve these bounds to $0.6421\leq\beta\leq0.8810$. Using importance sampling, we further show that $0.6536\leq \beta\leq 0.8749$ holds with probability at least $1-2\times 10^{-4}$. Here the probability is taken with respect to the randomness of the sampling procedure. 
	}
	

	\RUNTITLE{New Bounds for the Euclidean TSP Constant}
	\TITLE{New Bounds for the Euclidean TSP Constant}

	\ARTICLEAUTHORS{
    \AUTHOR{Zhuolun Dong$^{1}$, Junyu Cao$^{1}$}
    \AFF{$^{1}$McCombs School of Business, The University of Texas at Austin. \\ \EMAIL{dongzhuolun@utexas.edu, junyu.cao@mccombs.utexas.edu}.} 
    } 

    \RUNAUTHOR{Zhuolun Dong, Junyu Cao}
	
	\maketitle
	
	\section{Introduction}\label{sec:introduction}
	
	Let $X_1,\ldots,X_n$ be independent uniform points in the unit square. Denote the length of their shortest Euclidean traveling salesman tour by $\TSP(X_1,\ldots,X_n)$. The Beardwood-Halton-Hammersley (BHH) Theorem \citep{BHH1959} estimates this length in the asymptotic regime, as follows.
	\begin{theorem}[BHH Theorem]\label{thm1}
		There exists a universal constant $\beta$ such that
		\begin{align}\label{eq:bhh}
			\lim_{n\rightarrow\infty}\frac{\TSP(X_1,\ldots,X_n)}{\sqrt n}=\beta
		\end{align}
		holds almost surely. Moreover, $\beta$ satisfies $0.625\leq \beta\leq 0.9212$.
	\end{theorem}
	We refer readers to \cite{redmond1994limit} and \cite{daganzo2005logistics} for detailed proofs. Despite extensive progress on exact and heuristic TSP algorithms \citep{cook2010fifty,applegate2003chained,asadpour2017log}, the exact value of $\beta$ remains unknown. Recent numerical studies \citep{ApplegateEtAl2006,applegate2010using} show that the TSP constant $\beta\approx 0.7124$ empirically. So far, the tightest bounds established are $0.6277\leq\beta\leq0.90367$, due to \citet{GaudioJaillet2020} and \citet{GaudioGuan2026}.

	In this paper, we establish tighter upper and lower bounds:
	\begin{theorem}[Main result]\label{thm:main}
		The Euclidean TSP constant $\beta$ satisfies
		\begin{align}\label{eq:main-det}
			0.6421\leq \beta\leq 0.8810.
		\end{align}
		In addition, through importance sampling, we obtain that the following inequality holds with probability at least $1-2\times 10^{-4}$:
		\begin{align}\label{eq:main-hp}
			0.6536\leq \beta\leq 0.8749,
		\end{align}
		where the probability is taken with respect to the randomness of the sampling procedure.
		
	\end{theorem}
	
	Regarding the lower bound, we relate tour length to the connectivity of random geometric graphs. This yields a new lower bound on $\beta$ in the form of an infinite series of integrals. We evaluate the first few terms by a numerical method that partitions the integration domain into subregions and bounds the integrand conservatively on each.
	
	For the upper bound, we refine the classical strip construction. We divide the square into horizontal strips and process the points of each strip in blocks of size five. In each block, we establish an upper bound for the expected length and choose the visiting order that minimizes this value. The resulting expectation is an eleven-dimensional integral, which we bound by a numerical method.

	The remainder of the paper is organized as follows. Section~\ref{sec:literature} reviews the relevant literature. Section~\ref{sec:lower} proves the lower bounds. Section~\ref{sec:upper} proves the upper bounds.

	\section{Literature Review}\label{sec:literature}
	
	\citet{BHH1959} established the BHH Theorem and obtained the bounds $0.625 \le \beta \le 0.9212$. The same limiting behavior is also observed in a range of related Euclidean functionals, including the minimum spanning tree, the Steiner tree, and minimum weight matching \citep{Steele1997, Yukich1998}. Determining the value of these constants plays a central role in continuous approximation analysis of logistics systems \citep{daganzo2005logistics}. However, no closed form has been found for the TSP constant $\beta$.  Currently, the best bounds for $\beta$ are $0.6277 \le \beta \le 0.90367$. \citet{Steinerberger2015} later sharpened the BHH upper bound and proposed a refinement of the BHH lower-bound argument, which \citet{GaudioJaillet2020} corrected to obtain the lower bound $0.6277$. The upper bound comes from a line of work that optimizes the strip construction: \citet{CarlssonYu2025} allow swapping the visiting sequence of neighboring points in the strip, and \citet{GaudioGuan2026} further allow swapping points in adjacent strips.

	\section{Derivation of the Lower Bound}\label{sec:lower}
	
	In this section, we prove the lower bounds in Theorem~\ref{thm:main}. We first show that we can replace the $n$ points by a unit-intensity Poisson process and then derive a lower bound from path covers of the Poisson points. In Section~\ref{subsec:lower-components}, we establish a lower-bound formula for the TSP constant. Sections~\ref{subsec:lower-det} and~\ref{subsec:lower-hp} derive a deterministic and probabilistic bound, respectively.
	
	\subsection{Preliminaries}\label{subsec:lower-graph}
	
	To begin with, we note that we can replace the $n$ points by a Poisson process, as follows:

	\begin{lemma}[\citet{BHH1959,Steele1997}]\label{lem:poisson-unit}
		Let $\cP_n$ be a Poisson process with intensity $n$ on the unit square $[0,1]^2$, and let $\TSP(\cP_n)$ be the length of the TSP tour among all the points. Then,
		\begin{align}\label{eq:poisson-unit}
			\lim_{n\to\infty}\frac{\E[\TSP(\cP_n)]}{\sqrt n}=\beta.
		\end{align}
	\end{lemma}
	In this paper, we use the equivalent formulation of a unit-intensity Poisson process in an expanding square. Let $\cP$ be a unit-intensity Poisson process on the plane and let $Q_t=[0,t]^2$. Under the scaling $x\mapsto x/t$, $\cP\cap Q_t$ becomes a Poisson process of intensity $t^2$ on the unit square, while every Euclidean length is divided by $t$. Therefore, we have
	\begin{align}\label{eq:poisson}
		\lim_{t\to\infty}\frac{\E[\TSP(\cP\cap Q_t)]}{t^2}=\lim_{t\to\infty}\frac{\E[\TSP(\cP_{t^2})]}{t}=\beta.
	\end{align}

	We next relate tour length to the connectivity of a graph defined by a distance threshold. We introduce the following definitions. 
	\begin{definition}
		For a finite point set $X\subset\R^2$ and $r\geq0$, let $G_r(X)$ be the undirected graph that joins two points when their Euclidean distance is at most $r$. That is, $G_r(X)=(X,E)$, where
		\begin{align*}
			E=\left\{\{x_1,x_2\}:x_1,x_2\in X,\ x_1\neq x_2,\ \|x_1-x_2\|_2\leq r\right\}.
		\end{align*}
	\end{definition}
	
	\begin{definition}
		For a finite graph $G$, let $p(G)$ be the minimum number of vertex-disjoint paths (i.e., paths that do not share vertices) whose vertices cover the vertex set of $G$. A single vertex is allowed as a path. 
	\end{definition}
	
	\begin{definition}
		For a finite point set $X\subset\R^2$ with $|X|\geq3$, let $H$ be a Hamiltonian cycle on $X$. For $r\geq0$, let $N_H(r)$ be the number of edges of $H$ longer than $r$.
	\end{definition}
	
	\begin{definition}
		For a graph $G$, let $\cC(G)$ denote its collection of connected components. Each component is a maximal connected subgraph. In particular, $\cC(G)=\varnothing$ for the empty graph. Let $\tau(G)$ denote the number of spanning trees of $G$.
	\end{definition}
	Deleting edges longer than $r$ breaks a tour into paths, each contained in one component of $G_r(X)$. In Figure~\ref{fig:lower-construction}, we provide an illustration for this procedure. Building on the above definitions, we then establish the following result, which connects the length of a Hamiltonian cycle $H$ and $N_H(r)$.
	
	\begin{lemma}\label{lem:pathcover}
		For every finite $X$ with $|X|\geq3$, every Hamiltonian cycle $H$ on $X$, and every $r\geq0$,
		\begin{align}\label{eq:layer-path}
			\ell(H)=\int_0^\infty N_H(r)\d r,\quad N_H(r)\geq\sum_{C\in\cC(G_r(X))}p(C)-1,
		\end{align}
		where $\ell(H)$ denotes the length of $H$.
	\end{lemma}
	
	\proof{Proof.}
	Let $e_1,\ldots,e_n$ be the edge lengths of $H$. Since $e_i=\int_0^\infty\mathbbm{1}\{r<e_i\}\d r$, summing over the edges yields
	\begin{align*}
		\ell(H)=\sum_{i=1}^n e_i=\sum_{i=1}^n\int_0^\infty\mathbbm{1}\{r<e_i\}\d r=\int_0^\infty\sum_{i=1}^n\mathbbm{1}\{r<e_i\}\d r=\int_0^\infty N_H(r)\d r,
	\end{align*}
	which proves the first identity.
	
	For the second inequality, if $N_H(r)\geq1$, deleting the $N_H(r)$ long edges leaves exactly $N_H(r)$ vertex-disjoint paths. Note that every retained edge belongs to $G_r(X)$ because its length is no more than $r$, so any path must lie in a component $C$. Therefore, we obtain $N_H(r)\geq\sum_{C\in\cC(G_r(X))}p(C)$. If $N_H(r)=0$, then $G_r(X)$ is connected and contains the cycle $H$, so $\sum_{C\in\cC(G_r(X))}p(C)=1$, which completes the proof.
	\Halmos
	\endproof

\begin{figure}[t]
\centering
\FIGURE
{\scalebox{1.1}{\begin{tikzpicture}[
  x=1.1cm, y=1.1cm,
  font=\footnotesize,
  tour/.style={
    draw=tourblue, line width=1.1pt,
    line cap=round, line join=round
  },
  deleted/.style={
    draw=delorange, line width=0.9pt,
    dash pattern=on 3pt off 2pt
  },
  gedge/.style={draw=black!40, line width=0.5pt},
  pt/.style={
    circle, fill=black, inner sep=0pt, minimum size=3.2pt
  },
  ring/.style={
    circle, draw=tourblue, line width=0.8pt,
    inner sep=0pt, minimum size=8pt
  }
]
  \coordinate (ta) at (1.90,3.36);
  \coordinate (tb) at (1.40,4.06);
  \coordinate (tc) at (2.24,4.24);
  \coordinate (p1) at (3.44,4.56);
  \coordinate (p2) at (4.24,4.08);
  \coordinate (p3) at (5.12,4.44);
  \coordinate (p4) at (6.00,4.10);
  \coordinate (c)  at (5.00,2.20);
  \coordinate (y1) at (5.40,2.8928);
  \coordinate (z1) at (5.00,1.40);
  \coordinate (x1) at (4.20,2.20);
  \coordinate (S)  at (2.40,1.00);
  \draw[gedge] (ta) -- (tc);
  \draw[gedge] (c) -- (x1);
  \draw[deleted] (tc) -- (p1);
  \draw[deleted] (p4) -- (y1);
  \draw[deleted] (z1) -- (S);
  \draw[deleted] (S) -- (x1);
  \draw[deleted] (x1) -- (ta);
  \draw[tour] (ta) -- (tb) -- (tc);
  \draw[tour] (p1) -- (p2) -- (p3) -- (p4);
  \draw[tour] (y1) -- (c) -- (z1);
  \foreach \q in {ta,tb,tc,p1,p2,p3,p4,c,y1,z1,x1,S}
    \node[pt] at (\q) {};
  \node[ring] at (x1) {};
  \node[ring] at (S) {};
  \node[anchor=east]  at (1.52,3.40) {$p=1$};
  \node[anchor=south] at (4.68,4.72) {$p=1$};
  \node[anchor=west]  at (5.18,2.02) {$p=2$};
  \node[anchor=north] at (2.40,0.78) {$p=1$};
  \draw[tour] (0.50,-0.40) -- (0.95,-0.40)
    node[right=2pt, text=black] {kept tour edge};
  \draw[deleted] (4.00,-0.40) -- (4.45,-0.40)
    node[right=2pt, text=black] {deleted tour edge};
  \draw[gedge] (0.50,-0.98) -- (0.95,-0.98)
    node[right=2pt, text=black] {other edge of $G_r(X)$};
  \node[ring] at (4.225,-0.98) {};
  \node[pt] at (4.225,-0.98) {};
  \node[right=2pt] at (4.45,-0.98) {single-vertex path};
\end{tikzpicture}}}
{An illustration of the lower bound construction.
\label{fig:lower-construction}}
{Deleting the five edges longer than $r$ from a Hamiltonian
cycle leaves five paths in four connected components of
$G_r(X)$. Labels indicate the minimum numbers of
vertex-disjoint paths $p(C)$ for the connected components.}
\end{figure}

	Lemma \ref{lem:pathcover} allows us to connect the tour length with $p(C)$ for $C\in\cC(G_r(X))$. However, exactly computing $p(C)$ would be costly. We instead define and use a  lower bound $q$, as follows. 
	
	\begin{definition}
		For a finite graph $G$, write $d_G(v)$ for the degree of $v$ and $n_1(G)$ for the number of vertices of degree one. For $S\subseteq V(G)$, let $G-S$ be the subgraph of $G$ induced by $V(G)\setminus S$, that is, the graph on the vertex set $V(G)\setminus S$ that keeps every edge of $G$ with both endpoints in $V(G)\setminus S$. Let $\operatorname{cc}(G-S)=|\cC(G-S)|$, with $\operatorname{cc}(\varnothing)=0$. Write $(a)_+=\max\{a,0\}$. For connected $G$, define
		\begin{align}\label{eq:Dqg}
			D(G)=\frac12\sum_{v\in V(G)}(2-d_G(v))_+,\quad q(G)=\max\left\{1,\left\lceil\frac{n_1(G)}2\right\rceil,\max_{S\subseteq V(G)}\bigl(\operatorname{cc}(G-S)-|S|\bigr)\right\},
		\end{align}
		and let $g(G)=q(G)-D(G)$. For disconnected $G$, we set $g(G)=0$ instead.
	\end{definition}
	Then we obtain the following property on $p(G)$ and $q(G)$, which ensures that $q$ is a lower bound of $p$.
	
	\begin{lemma}\label{lem:deficiency}
		If $G$ is connected, then $p(G)\geq q(G)\geq D(G)$ and therefore $0\leq g(G)\leq p(G)-D(G)$.
	\end{lemma}
	
	\proof{Proof.}
	For each path $P$ in a path cover, an internal vertex has degree at least two in $G$, and each endpoint of a nonsingleton path contributes at most one to $(2-d_G(v))_+$. A singleton contributes at most two. Therefore, for each path $P$ in the cover, we have $\frac12\sum_{v\in V(P)}(2-d_G(v))_+\leq1$. Summing over a minimum path cover, we obtain $p(G)\geq D(G)$.
	
	Next we prove that $p(G)\geq q(G)$. Note that $p(G)\geq1$. Moreover, each path contains at most two leaves because every internal vertex has degree no less than two. We therefore obtain $p(G)\geq\lceil n_1(G)/2\rceil$. Further, for any $S\subseteq V(G)$, if the vertices of $S$ are deleted from a collection of $p(G)$ paths, each deletion increases the number of nonempty path pieces by at most one. Thus at most $p(G)+|S|$ pieces remain in the graph $G-S$, while every connected component of $G-S$ contains at least one of them. Hence $p(G)+|S|\geq\operatorname{cc}(G-S)$. Taking the maximum over the set $S$ yields $p(G)\geq q(G)$.
	
	Finally, for a connected graph with at least two vertices, there exists no isolated vertex. Only leaves have positive contributions to $D(G)$, so $D(G)=n_1(G)/2\leq\lceil n_1(G)/2\rceil\leq q(G)$. When the graph contains only a single vertex, the result is immediate since $p(G)=q(G)=D(G)=1$.
	\Halmos
	\endproof

	\subsection{The lower bound}\label{subsec:lower-components}

	By Lemmas~\ref{lem:pathcover} and~\ref{lem:deficiency}, at every threshold $r$ a tour has at least $\sum_{C\in\cC(G_r(X))}\bigl(D(C)+g(C)\bigr)-1$ edges longer than $r$. Integrating this inequality over $r$ yields a lower bound for the tour length, which can be translated into a lower bound for $\beta$. It suffices to bound the contributions of $D(C)$ and $g(C)$. We fix $k\geq3$, set $y_1=(0,0)$, and write $\boldsymbol y=(y_1,\ldots,y_k)$, where $y_i\in \R^2$. For $x\in\R^2$ and $s\geq0$, let $B(x,s)=\{u\in\R^2:\|u-x\|_2\leq s\}$ denote the closed disk of radius $s$ centered at $x$, and define
	\begin{align}\label{eq:A}
		A(\boldsymbol{y})=\mu\left(\bigcup_{i=1}^k B(y_i,1)\right),
	\end{align}
	where $\mu(\cdot)$ denotes planar area. Then we establish the following lower bound for $\beta$.

	\begin{theorem}\label{thm:lower-master}
		For $r>0$, define
		\begin{align}\label{eq:rho}
			\rho_k(r)=\frac1{k!}\int_{(\R^2)^{k-1}}\mathbbm{1}\{G_r(\boldsymbol{y})\text{ is connected}\}g(G_r(\boldsymbol{y}))e^{-\mu(\cup_{i=1}^kB(y_i,r))}\d y_2\cdots\d y_k.
		\end{align}
		Then, for every finite integer $K\geq3$, we have
		\begin{align}\label{eq:lower-master}
			\beta\geq\frac58+\sum_{k=3}^KC_k,
		\end{align}
		where
		\begin{align}\label{eq:Ck}
			C_k=\int_0^\infty\rho_k(r)\d r =\frac{\Gamma(k-\tfrac12)}{2k!}\int_{(\R^2)^{k-1}}\frac{g(G_1(\boldsymbol{y}))}{A(\boldsymbol{y})^{k-1/2}}\d y_2\cdots\d y_k.
		\end{align}
		Here $\Gamma(a)=\int_0^\infty u^{a-1}e^{-u}\d u$ for $a>0$.
	\end{theorem}

    In Inequality \eqref{eq:lower-master}, the term $5/8$ refers to the contribution of $D(C)$, and the last term $\sum C_k$ collects the additional contributions of $g(C)$. In particular, $C_k$ refers to the contribution of connected components of size $k$. Then we provide the proof of Theorem \ref{thm:lower-master}.

	\proof{Proof.}
	Fix $R,t>0$. Let $X_t=\cP\cap Q_t$, and set $T_t=\TSP(X_t)$, where the tour length of a set with at most one point is zero and that of a two-point set is twice the distance between the points. We first integrate over the finite interval $[0,R]$. When $|X_t|\geq3$, Lemmas~\ref{lem:pathcover} and~\ref{lem:deficiency} imply
	\begin{align}\label{eq:lower-prelimit}
		T_t\geq\int_0^R\left(\sum_{C\in\cC(G_r(X_t))}p(C)-1\right)\d r \geq \int_0^R\left(\sum_{C\in\cC(G_r(X_t))}(D(C)+g(C))-1\right)\d r.
	\end{align}
	When $|X_t|\leq2$, the right-hand side of \eqref{eq:lower-prelimit} is nonpositive if $X_t$ has at most one point and is at most the distance between the two points otherwise, so \eqref{eq:lower-prelimit} holds for every realization of $X_t$. 
    
    We bound the integral of $D(C)$ and $g(C)$ separately. We consider the first term. For $x\in X_t$, define $d_r^\infty(x)=\#((\cP\setminus\{x\})\cap B(x,r))$ and $d_r^t(x)=\#((X_t\setminus\{x\})\cap B(x,r))$. Then $d_r^\infty(x)\geq d_r^t(x)$, and hence
	\begin{align}\label{eq:degree-window}
		\sum_{C\in\cC(G_r(X_t))}D(C) = \frac12\sum_{x\in X_t}(2-d_r^t(x))_+ \geq\frac12\sum_{x\in X_t}(2-d_r^\infty(x))_+.
	\end{align}
	Applying the Mecke equation \citep[Theorem~4.1]{LastPenrose2018} to the right-hand side yields
	\begin{align}\label{eq:degree-Mecke}
		\E\left[\sum_{C\in\cC(G_r(X_t))}D(C)\right]\geq\frac12\int_{Q_t}\E\left[(2-Z_r)_+\right]\d x:=t^2d(r),
	\end{align}
	where $Z_r\sim\Poi(\pi r^2)$, and $d(r):=\frac12\E(2-Z_r)_+$.

	For the second term, we count connected components of the full-plane graph that lie entirely in $Q_t$. These components are also components of $G_r(X_t)$. Fix $k\in\{3,\ldots,K\}$, define $a=(k-1)r$, and let $Q_t^{\ominus a}=\{x\in Q_t:\operatorname{dist}(x,\partial Q_t)\geq a\}$. For an ordered $k$-tuple of distinct Poisson points $\boldsymbol x=(x_1,\ldots,x_k)$, define
    \resizebox{\linewidth}{!}{$\displaystyle
        F_{k,t,r}(\boldsymbol x,\cP)=\mathbbm{1}\{x_1\in Q_t^{\ominus a}\}\mathbbm{1}\{G_r(\boldsymbol x)\text{ is connected}\}g(G_r(\boldsymbol x))\times\mathbbm{1}\left\{(\cP\setminus\{x_1,\ldots,x_k\})\cap\bigcup_{i=1}^kB(x_i,r)=\varnothing\right\}.
    $}
	The last indicator implies that the selected points $\boldsymbol{x}$ form a connected component of $G_r(\cP)$. Now we attempt to connect $g(C)$ and $F_{k,t,r}(\boldsymbol x,\cP)$.

	Note that every vertex can be reached from $x_1$ by at most $k-1$ edges, and each edge has length at most $r$. We conclude that its Euclidean distance from $x_1$ is at most $(k-1)r$. Thus $x_1\in Q_t^{\ominus a}$ ensures that all selected points lie in $Q_t$. Recall that each unordered connected component of size $k$ has $k!$ orderings. For an isolated $k$-point connected component $C$, define $m_t(C):=\#\bigl(C\cap Q_t^{\ominus a}\bigr)$ (then $m_t(C)\leq k$). Each of these $m_t(C)$ points can appear as the first entry of the ordered tuple, and the remaining $k-1$ points can then be ordered in $(k-1)!$ ways. Therefore, for every realization of the Poisson process $\cP$, we obtain
	\begin{align}\label{eq:component-ordering}
		\frac1{k!}\sum_{\boldsymbol{x}\in\cP_{\neq}^k} F_{k,t,r}(\boldsymbol{x},\cP) &= \sum_{\substack{C\text{ isolated in }G_r(\cP)\\ C\subseteq Q_t, |C|=k}} \frac{m_t(C)}{k}g(C) \leq \sum_{\substack{C\text{ isolated in }G_r(\cP)\\ C\subseteq Q_t, |C|=k}} g(C),
	\end{align}
	where $\cP_{\neq}^k = \{(x_1,\ldots,x_k)\in \cP^k: x_i\neq x_j \text{ for all } i\neq j\}$ is the set of ordered $k$-tuples of distinct Poisson points.

	For fixed $x_1,\ldots,x_k$, the probability that the Poisson process contains no additional point in $\bigcup_{i=1}^k B(x_i,r)$ is $\exp(-\mu(\bigcup_{i=1}^k B(x_i,r)))$. The multivariate Mecke equation \citep[Theorem~4.4]{LastPenrose2018} therefore establishes
	\begin{align}\label{eq:Mecke-integral}
		&\frac{1}{k!}\E\left[\sum_{(x_1,\ldots,x_k)\in\cP_{\neq}^{k}}F_{k,t,r}\bigl((x_1,\ldots,x_k),\cP\bigr)\right] \nonumber\\ =&\frac1{k!}\int_{Q_t^{\ominus a}}\int_{(\R^2)^{k-1}}\mathbbm 1\{G_r(\boldsymbol{x})\text{ is connected}\}g\bigl(G_r(\boldsymbol{x})\bigr)e^{-\mu\left(\bigcup_{i=1}^kB(x_i,r)\right)}\d x_2\cdots\d x_k \d x_1
	\end{align}
	Now fix $x_1$ and consider the inner integral over $x_2,\ldots,x_k$. Its value is the same for every $x_1$, because shifting all $k$ points by the same vector does not change their distances or the area of the disk union. After translating $x_1$ to the origin, we find that the inner integral is exactly $k!\rho_k(r)$ by Equation \eqref{eq:rho}. Since $\mu(Q_t^{\ominus a})=(t-2(k-1)r)_+^2$, we can use \eqref{eq:component-ordering} and establish that
	\begin{align}\label{eq:component-Mecke}
		\E\Bigg[\sum_{\substack{C\in\cC(G_r(\cP))\\ C\subseteq Q_t, |C|=k}}g(C) \Bigg] \geq  \frac{1}{k!}\E\left[\sum_{(x_1,\ldots,x_k)\in\cP_{\neq}^{k}}F_{k,t,r}\bigl(\boldsymbol{x},\cP\bigr)\right] = \bigl(t-2(k-1)r\bigr)_+^2\rho_k(r).
	\end{align}

	We now combine the integrals of $g$ and $D$.  Lemma~\ref{lem:deficiency} gives $p(C)\geq D(C)+g(C)$ for every component of $G_r(X_t)$. Because $g\geq0$, we may retain only full-plane components of sizes $3,\ldots,K$ in the correction term. Taking expectations in \eqref{eq:lower-prelimit} and applying \eqref{eq:degree-Mecke} and \eqref{eq:component-Mecke} yields
	\begin{align}\label{eq:lower-window}
		\frac{\E [T_t]}{t^2} &\geq\frac1{t^2}\E\left[\int_0^R\sum_{C\in\cC(G_r(X_t))}\bigl(D(C)+g(C)\bigr)\d r\right]-\frac{R}{t^2}\nonumber\\
		&\geq\int_0^Rd(r)\d r+\sum_{k=3}^K\int_0^R\left(1-\frac{2(k-1)r}{t}\right)_+^2\rho_k(r)\d r-\frac{R}{t^2}.
	\end{align}
	For fixed $R$ and $K$, the boundary factors converge uniformly to one on $[0,R]$ as $t\to\infty$. Each $\rho_k$ is bounded on this interval: connectivity of $G_r(\boldsymbol y)$ requires $\|y_i\|_2\leq(k-1)R$, while $0\leq g(G)\leq p(G)\leq k$. Thus dominated convergence and \eqref{eq:poisson} give
	\begin{align*}
		\beta = \lim_{t\rightarrow\infty}\frac{\E T_t}{t^2}&\geq \lim_{t\rightarrow\infty} \left(\int_0^Rd(r)\d r+\sum_{k=3}^K\int_0^R\left(1-\frac{2(k-1)r}{t}\right)_+^2\rho_k(r)\d r-\frac{R}{t^2}\right) \\&\geq  \int_0^Rd(r)\d r+\sum_{k=3}^K\int_0^R\rho_k(r)\d r.
	\end{align*}
	Letting $R\to\infty$ and using the monotone convergence theorem yields
	\[\beta\geq\int_0^\infty d(r)\d r+\sum_{k=3}^K\int_0^\infty\rho_k(r)\d r = \frac58 +\sum_{k=3}^K\int_0^\infty\rho_k(r)\d r, \]
	which proves \eqref{eq:lower-master}. The equation holds because
	\begin{align*}
		\int_0^\infty d(r)\d r &= \int_0^{\infty}\left(\Pr(Z_r=0)+\frac12\Pr(Z_r=1) \right) \d r\\
        &=\int_0^\infty e^{-\pi r^2}\left(1+\frac{\pi r^2}{2}\right)\d r=\frac12+\frac12\int_0^\infty e^{-\pi r^2}\pi r^2\d r=\frac58.
	\end{align*}
	
	It remains to derive the second equation in \eqref{eq:Ck}. By definition, we see that
	\begin{align*}
		C_k=\frac1{k!}\int_0^\infty\int_{(\R^2)^{k-1}}\mathbbm 1\{G_r(\boldsymbol{z})\text{ is connected}\}g\bigl(G_r(\boldsymbol{z})\bigr)\exp\left(-\mu\left(\bigcup_{i=1}^kB(z_i,r)\right)\right)\d z_2\cdots\d z_k\d r,
	\end{align*}
	where $\boldsymbol{z}= \big((0,0),z_2,\ldots,z_k\big)$. For $r>0$, set $z_i=ry_i$. Then $\{z_i,z_j\}$ is an edge of $G_r(\boldsymbol z)$ if and only if $\|y_i-y_j\|_2\leq1$. The disk-union area becomes $r^2A(\boldsymbol y)$, and the Jacobian for the $k-1$ free planar coordinates is $r^{2k-2}$. Hence
	\begin{align}\label{eq:Ck-after-scaling}
		C_k=\frac1{k!}\int_0^\infty\int_{(\R^2)^{k-1}}\mathbbm 1\{G_1(\boldsymbol{y})\text{ is connected}\}g\bigl(G_1(\boldsymbol{y})\bigr)r^{2k-2}e^{-r^2A(\boldsymbol{y})}\d y_2\cdots\d y_k\d r.
	\end{align}
	The integrand is nonnegative, so Tonelli's theorem permits exchanging the integrals. Evaluating the inner integral with $u=r^2A(\boldsymbol y)$ yields
	\begin{align*}
		C_k&=\frac1{k!}\int_{(\R^2)^{k-1}}\mathbbm 1\{G_1(\boldsymbol{y})\text{ is connected}\}g\bigl(G_1(\boldsymbol{y})\bigr)\left(\int_0^\infty r^{2k-2}e^{-r^2A(\boldsymbol{y})}\d r\right)\d y_2\cdots\d y_k\\
		&=\frac1{k!}\int_{(\R^2)^{k-1}}\mathbbm 1\{G_1(\boldsymbol{y})\text{ is connected}\}g\bigl(G_1(\boldsymbol{y})\bigr)\left(\frac12\Gamma\left(k-\frac12\right)A(\boldsymbol{y})^{-(k-1/2)}\right)\d y_2\cdots\d y_k\\
		&=\frac1{k!}\int_{(\R^2)^{k-1}}g\bigl(G_1(\boldsymbol{y})\bigr)\left(\frac12\Gamma\left(k-\frac12\right)A(\boldsymbol{y})^{-(k-1/2)}\right)\d y_2\cdots\d y_k,
	\end{align*}
	which completes the proof. The last equation holds because $g(G)=0$ for any disconnected graph $G$.
	\Halmos
	\endproof
	
	\subsection{Deterministic lower bound}\label{subsec:lower-det}
	Theorem \ref{thm:lower-master} establishes a lower bound of $\beta$. Direct evaluation is infeasible because the dimension of $C_k$ grows with $k$. It therefore suffices to derive a lower bound of $C_k$. In particular, for $k=3,4,5$, we use the computer-assisted method by partitioning the integration region into disjoint rectangles and bounding the integral from below over each rectangle.

	\subsubsection{Computer-assisted lower bounds for $C_k$.}
	We begin with reducing the integration region. The integrand in \eqref{eq:Ck} is invariant under rotations and reflections of all points. Write $y_2=(r_2\cos\theta,r_2\sin\theta)$, where $r_2\geq0$ and $\theta\in[0,2\pi)$. For each $\theta$, rotate all points by $-\theta$ and write the resulting coordinates as $\boldsymbol y=\bigl((0,0),(r_2,0),(\xi_3,\eta_3),\ldots,(\xi_k,\eta_k)\bigr)$. Then we obtain
	\begin{align*}
		C_k &=\frac{\Gamma(k-\tfrac12)}{2k!}\int_0^\infty\int_0^{2\pi}\int_{\R^{2k-4}}r_2g(G_1(\boldsymbol y))A(\boldsymbol y)^{-(k-1/2)}\d \xi_3\d \eta_3\cdots\d \xi_k\d \eta_k\d\theta\d r_2 \\
		&=\frac{\pi\Gamma(k-\tfrac12)}{k!}\int_0^\infty\int_{\R^{2k-4}}r_2g(G_1(\boldsymbol y))A(\boldsymbol y)^{-(k-1/2)}\d \xi_3\d \eta_3\cdots\d \xi_k\d \eta_k\d r_2.
	\end{align*}
	In addition, reflecting all points across the horizontal axis changes every $\eta_j$ to $-\eta_j$ and preserves the integrand except on the measure-zero set $\eta_3=0$. Therefore, we may restrict the integration region to $\eta_3\geq0$ and yield:
	\begin{align}\label{eq:Ck-reduced}
		C_k=\frac{2\pi\Gamma(k-\tfrac12)}{k!}\int_{[0,\infty)\times\R\times[0,\infty)\times\R^{2k-6}}r_2g(G_1(\boldsymbol y))A(\boldsymbol y)^{-(k-1/2)}\d r_2\d\xi_3\d\eta_3\cdots\d\xi_k\d\eta_k.
	\end{align}
    Because the integrand in \eqref{eq:Ck-reduced} is nonnegative, we may restrict the integration to a region $\Omega_k \subseteq[0,\infty)\times\R\times[0,\infty)\times\R^{2k-6}$ (specified below) to obtain a lower bound on $C_k$. For computational convenience, we consider the contribution of $C_k$ for $k=3,4,5$. For each $k$, we partition $\Omega_k$ into cubes of side length $h$ with disjoint interiors (called subregions) and bound the integral over each subregion from below.

	Fix a subregion $\cQ$ whose $r_2$ interval is $[nh,(n+1)h]$, where $n$ is a non-negative integer. To determine which graph edges can occur in $\cQ$, we define the following for each $i<j$:
	\begin{align*}
		d_{ij}^{-}(\mathcal Q)&=\min_{\boldsymbol y\in\mathcal Q}\|y_i-y_j\|_2 = \min_{\boldsymbol y\in\mathcal Q}\|(\xi_i,\eta_i)-(\xi_j,\eta_j)\|_2,\\
        d_{ij}^{+}(\mathcal Q)&=\max_{\boldsymbol y\in\mathcal Q}\|y_i-y_j\|_2=\max_{\boldsymbol y\in\mathcal Q}\|(\xi_i,\eta_i)-(\xi_j,\eta_j)\|_2.
	\end{align*}
	These values can be computed exactly from the coordinate intervals. The edge $\{i,j\}$ is present throughout $\cQ$ (i.e., for any $(r_2,\xi_3,\eta_3,\ldots,\xi_k,\eta_k)\in \cQ$) if $d_{ij}^{+}(\cQ)\leq1$ (we call it \emph{always-present}), absent throughout $\cQ$ if $d_{ij}^{-}(\cQ)>1$ (we call it \emph{always-absent}), and otherwise undetermined. Let $\mathcal G(\mathcal Q)$ be the set of graphs that contain every always-present edge, contain no always-absent edge, and may contain any subset of the undetermined edges. Then $G_1(\boldsymbol y)\in\mathcal G(\mathcal Q)$ holds for every $\boldsymbol y\in\mathcal Q$. Define
	\begin{align}\label{eq:g-box}
		\underline g(\mathcal Q)=\min_{G\in\mathcal G(\mathcal Q)}g(G).
	\end{align}

    We have $g(G_1(\boldsymbol y))\geq\underline g(\cQ)$ for every $\boldsymbol y\in\cQ$; in particular, $\underline g(\cQ)=0$ whenever $\mathcal G(\mathcal Q)$ contains a disconnected graph.

	It remains to bound $A(\boldsymbol y)$ from above. We obtain such a bound by estimating the area contribution of adding disks along a spanning tree. First, $c(d)$ is the area of one unit disk that lies outside another unit disk when their centers are a distance $d$ apart. We obtain the following property of $c(d)$.

	\begin{lemma}\label{lem:disk-area}
		For $0\leq d\leq1$,
		\begin{align}\label{eq:crescent}
			c(d)=2\arcsin\left(\frac d2\right)+\frac d2\sqrt{4-d^2} =\int_0^d\sqrt{4-u^2}\d u \leq P(d):=2d-\frac{d^3}{12}-\frac{d^5}{320}.
		\end{align}
		Moreover, both $c(d)$ and $P(d)$ are increasing on $[0,1]$.
	\end{lemma}
	
	\proof{Proof.}
    The overlap of two unit disks at distance $d$ has area $2\arccos(d/2)-(d/2)\sqrt{4-d^2}$; subtracting it from $\pi$ and using $\pi-2\arccos x=2\arcsin x$ gives the first identity in \eqref{eq:crescent}. Moreover,  differentiate $c(d)$ yields $c'(d)=\sqrt{4-d^2}>0$, which together with $c(0)=0$ gives the second identity and shows that $c$ is increasing. For $0\leq u\leq1$, we have $P'(u)=2-u^2/4-u^4/64>0$ and
    \[
    P'(u)^2-(4-u^2)=\frac{u^6}{128}+\frac{u^8}{4096}\geq0,
    \]
    so $P'(u)\geq\sqrt{4-u^2}=c'(u)>0$. Hence $P$ is increasing, and integrating over $[0,d]$ with $P(0)=c(0)=0$ gives $c(d)\leq P(d)$.
    \Halmos
    \endproof

	We are ready to bound $A(\boldsymbol{y})$. For a fixed subregion $\mathcal Q$, consider the graph on the points $y_1,\ldots,y_k$ that contains the edge $\{y_i,y_j\}$ whenever $d_{ij}^{+}(\mathcal Q)\leq1$. Then every edge of this graph is present in $G_1(\boldsymbol y)$ for every $\boldsymbol y\in\mathcal Q$. If these edges connect all points $y_1,\ldots,y_k$, they contain at least one spanning tree. Let $\mathcal T(\mathcal Q)$ denote the collection of all such spanning trees.

	Fix any $T\in\mathcal T(\mathcal Q)$ and root it at $y_1$. For each $y_j\neq y_1$, its parent is the unique neighbor of $y_j$ on the path from $y_j$ to $y_1$ in $T$. Its other neighbors in $T$, if any, are its children. We add the unit disks $B(y_j,1)$ in an order such that the disk centered at each parent is added before those centered at its children. The first disk has area $\pi$. If $y_i$ is the parent of $y_j$ in $T$, then adding $B(y_j,1)$ increases the union area by at most $c(\|y_i-y_j\|_2)$. Since $\|y_i-y_j\|_2\leq d_{ij}^{+}(\mathcal Q)\leq1$ and $c(\cdot)$ is increasing, applying Lemma~\ref{lem:disk-area} leads to the following inequality.
	\begin{align*} 
		A(\boldsymbol y) &\leq \pi+\sum_{\{y_i,y_j\}\in T}c\bigl(\|y_i-y_j\|_2\bigr) \leq \pi+\sum_{\{y_i,y_j\}\in T}c\bigl(d_{ij}^{+}(\mathcal Q)\bigr) \leq \pi+\sum_{\{y_i,y_j\}\in T}P\bigl(d_{ij}^{+}(\mathcal Q)\bigr), \quad \boldsymbol y\in\mathcal Q.
	\end{align*}
	Taking the minimum over all $T\in\mathcal T(\mathcal Q)$, we define
	\begin{align}\label{eq:A-box}
		\underline A(\mathcal Q) := \pi+\min_{T\in\mathcal T(\mathcal Q)} \sum_{\{y_i,y_j\}\in T}P\bigl(d_{ij}^{+}(\mathcal Q)\bigr).
	\end{align}
	Equipped with Equations \eqref{eq:A-box} and \eqref{eq:g-box}, we obtain the following result.
	
	\begin{lemma}\label{lem:box-lower}
		If the always-present edges on $\mathcal Q$ connect all points $y_1,\ldots,y_k$, then we have
		\begin{align}\label{eq:box-lower}
			\int_{\cQ}r_2g(G_1(\boldsymbol y))A(\boldsymbol y)^{-(k-1/2)}\d r_2\d \xi_3\d \eta_3\cdots\d \xi_k\d \eta_k
			\geq \frac{2n+1}{2}h^{2k-2}\underline g(\mathcal Q)\underline A(\mathcal Q)^{-(k-1/2)}.
		\end{align}
		
	\end{lemma}
	
	\proof{Proof.}
	For every $\boldsymbol y\in\mathcal Q$, the definition of $\underline g(\mathcal Q)$ and $\underline A(\mathcal Q)$ give $g(G_1(\boldsymbol y))\geq\underline g(\mathcal Q)$ and $A(\boldsymbol y)\leq \underline A(\mathcal Q).$ Therefore, we have $g(G_1(\boldsymbol y))A(\boldsymbol y)^{-(k-1/2)} \geq \underline g(\mathcal Q) \underline A(\mathcal Q)^{-(k-1/2)}.$ Applying this inequality yields the desired results:
	\begin{align*}
		\int_{\mathcal Q} r_2g(G_1(\boldsymbol y))A(\boldsymbol y)^{-(k-1/2)} \d r_2\d\xi_3\d\eta_3\cdots\d\xi_k\d\eta_k& \geq \underline g(\mathcal Q) \underline A(\mathcal Q)^{-(k-1/2)} \int_{\mathcal Q} r_2 \d r_2\d\xi_3\d\eta_3\cdots\d\xi_k\d\eta_k\\
		&= \underline g(\mathcal Q)\underline A(\mathcal Q)^{-(k-1/2)} \left(\int_{nh}^{(n+1)h}r_2\d r_2\right)h^{2k-4}\\
		&= \frac{2n+1}{2}h^{2k-2} \underline g(\mathcal Q) \underline A(\mathcal Q)^{-(k-1/2)},
	\end{align*}
	which completes the proof. \Halmos
	\endproof
	
	Multiplying \eqref{eq:box-lower} by $2\pi\Gamma(k-\tfrac12)/k!$ and summing over each subregion yields a lower bound for $C_k$. We only account for subregions whose always-present edges connect all vertices. We next specify the computational setting we used.

    For $k=3$, a connected graph is either a path or a triangle. Note that a path has $D(G)=q(G)=1$ and hence $g(G)=0$. Thus, the integrand is positive only when the three points form a triangle (which has $g(G)=1$). We consider the integration region $\Omega_3=[0,1]\times[-1,1]\times[0,1]$ in coordinates $(r_2,\xi_3,\eta_3)$. We use subregions of side length $1/1600$ and the preceding lower-bound rule. For $k=4$, we restrict the integration to the domain
    \[\Omega_4=\Big\{(r_2,\xi_3,\eta_3,\xi_4,\eta_4): r_2,\eta_3\in[0,2],\ \xi_3,\xi_4,\eta_4\in[-2,2]\Big\}, \]
    which we partition into subregions of side length $1/50$. For $k=5$, we restrict the integration to the domain 
    \[\Omega_5=\Big\{(r_2,\xi_3,\eta_3,\xi_4,\eta_4,\xi_5,\eta_5): r_2,\eta_3\in[0,3],\ \xi_3,\xi_4,\eta_4,\xi_5,\eta_5\in[-3,3]\Big\}.\]
    We partition this region into subregions of side length $1/10$ and, as a preliminary filter, retain only subregions whose always-present edges connect $y_1,\ldots,y_4$.

    We claim that $g(G_1(\boldsymbol y))>0$ implies $\|y_j\|_2\leq k-2$ for every $j=2,\ldots,k$. The graph is then connected. If a vertex had graph distance $k-1$ from $y_1$, a shortest path to it would contain all $k$ vertices, and any additional edge would shorten this path. The graph would therefore itself be a path, which implies $g(G_1(\boldsymbol y))=0$, a contradiction. Thus every vertex can be reached from $y_1$ by at most $k-2$ edges, each of length at most one. This yields $\|y_j\|_2\leq k-2$, which proves the claim. Since the integrand is nonnegative and, by the claim, vanishes unless $\|y_j\|_2\leq k-2$ for all $j$, we have
    \begin{align}\label{eq:Ck-box}
		C_k&=\frac{2\pi\Gamma(k-\tfrac12)}{k!}\int_{[0,\infty)\times\R\times[0,\infty)\times\R^{2k-6}}r_2g(G_1(\boldsymbol y))A(\boldsymbol y)^{-(k-1/2)}\d r_2\d\xi_3\d\eta_3\cdots\d\xi_k\d\eta_k\nonumber\\
        &\geq \frac{2\pi\Gamma(k-\tfrac12)}{k!}\int_{\Omega_k}r_2g(G_1(\boldsymbol y))A(\boldsymbol y)^{-(k-1/2)}\d r_2\d\xi_3\d\eta_3\cdots\d\xi_k\d\eta_k\nonumber\\
        &=\frac{2\pi\Gamma(k-\tfrac12)}{k!}\int_{\Omega_k\cap \{\boldsymbol y:\ \|y_j\|_2\leq k-2,\ j=2,\ldots,k\}}r_2g(G_1(\boldsymbol y))A(\boldsymbol y)^{-(k-1/2)}\d r_2\d\xi_3\d\eta_3\cdots\d\xi_k\d\eta_k.
	\end{align}

    Finally, we apply Lemma~\ref{lem:box-lower} for computing $C_3,C_4,C_5$. In all these computations, comparisons of distances with the threshold $1$ are carried out with rational squared distances, and the irrational quantities ($\Gamma(k-\tfrac12)$, $\pi$, and the square roots) are enclosed by rational numbers rounded in the direction that decreases the bound, so that every reported value is a valid lower bound. Substituting these results (their values are reported in Table~\ref{tab:numerics}) into Theorem~\ref{thm:lower-master} with $K=5$ yields \[\beta\geq 0.625+0.01113528859 + 0.005040573276  + 0.001015487669  > 0.64219>  0.6421 .\]

	\subsection{High-probability lower bound}\label{subsec:lower-hp}
	
	In Section~\ref{subsec:lower-det}, we divided the integration region into subregions and bounded the integrand from below. In this section we estimate $C_k$ by importance sampling. For each $k$, we construct a random variable $W_k$ that satisfies $\E[W_k]=C_k$, where the expectation is taken over the random points $Y_2,\ldots,Y_k$ (we use the uppercase letters to denote a random variable and the lowercase letters to denote their realization). Next, we take $n_k$ independent copies of $W_k$ and apply a concentration inequality to them. This yields a lower confidence bound for $C_k$, allowing us to obtain a lower bound on $\beta$ that holds with high probability.

	In particular, we generate $\boldsymbol{Y}=(Y_1,\ldots,Y_k)$ using a spanning tree. Cayley's formula implies that there are $k^{k-2}$ trees on $k$ labeled vertices. We first choose one of these trees uniformly at random, root it at vertex $1$, place $Y_1=(0,0)$, and direct every edge away from the root. Then we realize the locations of $Y_i$ according to this sampled tree. For each directed edge, we independently sample a displacement vector uniformly from the unit disk and place each child node by adding the displacement vector to the location of its parent node. By this construction, such a sampled tree belongs to $G_1(\boldsymbol{Y})$, so $G_1(\boldsymbol{Y})$ is connected.

	We next calculate the density function of $\boldsymbol{Y}$ generated through the above process. Fix a realization $\boldsymbol y$ with $y_1=(0,0)$. Then a tree $T$ can generate $\boldsymbol y$ under the above procedure if and only if $E(T)\subseteq E(G_1(\boldsymbol y))$ (equivalently, $G_1(\boldsymbol y)$ contains $T$). Hence, if $G_1(\boldsymbol y)$ is connected, exactly $\tau(G_1(\boldsymbol y))$ of the $k^{k-2}$ labeled trees can generate $\boldsymbol y$, where $\tau(\cdot)$ is the number of spanning trees in the graph. The edge displacements are independent. Conditional on the previously placed points, each child is uniformly distributed in the unit disk centered at its parent, with density $1/\pi$. As a result, the joint density is $\pi^{-(k-1)}$ everywhere; conditional on a particular tree $T$, the density of $(Y_2,\ldots,Y_k)$ at $(y_2,\ldots,y_k)$ is 
    $$f_{Y_2,\ldots,Y_k}(y_2,\ldots,y_k) = \pi^{-(k-1)}\mathbbm 1(E(T)\subseteq E(G_1(\boldsymbol{y}))).$$ 
    By enumerating all possible choices of $T$, we find that the density function of $\boldsymbol{Y}$ can be formulated as:
	\begin{align}\label{eq:tree-density}
		f_k(\boldsymbol y) &=\sum_{T} \pi^{-(k-1)}\mathbbm 1(E(T)\subseteq E(G_1(\boldsymbol{y}))) \cdot \frac{1}{k^{k-2}}  = \frac{\tau(G_1(\boldsymbol y))}{k^{k-2}\pi^{k-1}} \mathbbm{1}(G_1(\boldsymbol y)\text{ is connected}),
	\end{align}
	where the last equation holds because $G_1(\boldsymbol{y})$ is connected if and only if it contains at least one spanning tree.

	Now we are able to define $W_k$. Let
	\begin{align}\label{eq:Wk}
		W_k=\frac{\Gamma(k-\tfrac12)}{2k!}\pi^{k-1}k^{k-2}\frac{g(G_1(\boldsymbol Y))}{\tau(G_1(\boldsymbol Y))A(\boldsymbol Y)^{k-1/2}}.
	\end{align}
	By Equation \eqref{eq:tree-density}, we obtain
	\begin{align*}
		\E[W_k] &= \int_{(\R^2)^{k-1}} \frac{\Gamma(k-\tfrac12)}{2k!}\pi^{k-1}k^{k-2}\frac{g(G_1(\boldsymbol y))}{\tau(G_1(\boldsymbol y))A(\boldsymbol y)^{k-1/2}} f_k(\boldsymbol y) \d y_2\cdots\d y_k \\
		&= \frac{\Gamma(k-\tfrac12)}{2k!}\int_{(\R^2)^{k-1}}\frac{g(G_1(\boldsymbol y))}{A(\boldsymbol y)^{k-1/2}}\d y_2\cdots\d y_k =C_k.
	\end{align*}
	Here the integrand is defined to be zero wherever $f_k(\boldsymbol y)=0$. For each $k$, fix before sampling an integer $n_k\geq2$ and a clipping level $M_k>0$. Let $W_{k,1},\ldots,W_{k,n_k}$ be independent copies of $W_k$, $X_{k,i}=\min\{W_{k,i},M_k\}$, and define
	\[
	\overline X_k=\frac1{n_k}\sum_{i=1}^{n_k}X_{k,i},\quad s_k^2=\frac1{n_k-1}\sum_{i=1}^{n_k}(X_{k,i}-\overline X_k)^2.
	\]
	By definition, we have $0\leq X_{k,i}\leq M_k$ and $\E [X_{k,i}]\leq\E [W_k]=C_k$. Applying the Empirical Bernstein Inequality \citep{MaurerPontil2009} to $1-X_{k,i}/M_k$ yields,
	with probability at least $1-\delta_k$,
	\begin{align}\label{eq:ell-k}
		C_k\geq\ell_k:=\max\left\{0,\overline X_k- \sqrt{\frac{2s_k^2\log(2/\delta_k)}{n_k}} -\frac{7M_k\log(2/\delta_k)}{3(n_k-1)}\right\}.
	\end{align}
	With this inequality, we can establish the desired result.
	
	\begin{proposition}\label{prop:lower-hp}
		Let $0<\delta<1$ and set $\delta_k=\delta/8$. Then
		\begin{align}\label{eq:lower-hp}
			\Pr\left(\beta\geq\frac58+\sum_{k=3}^{10}\ell_k\right)\geq1-\delta.
		\end{align}
	\end{proposition}
	
	\proof{Proof.}
	For each $k$, the event $C_k\geq\ell_k$ fails with probability at most $\delta/8$. By the union bound, all eight inequalities hold simultaneously with probability at least $1-\sum_{k=3}^{10}\delta_k=1-\delta$. On this event, we have \[\beta\geq \frac58+\sum_{k=3}^{10}C_k\geq\frac58+\sum_{k=3}^{10}\ell_k,\]
	which completes the proof.
	\Halmos
	\endproof
	
	With Proposition \ref{prop:lower-hp}, we can obtain the high-probability lower bound by computing $\ell_k$. With $n_k=5\times 10^4$, $\delta=10^{-4}$, and the clipping levels in Table~\ref{tab:numerics}, we obtain 
    \[\frac58+\sum_{k=3}^{10}\ell_k\geq0.6536702342>0.6536.\] 
    Hence $\beta\geq 0.6536$ with probability at least $1-10^{-4}$.

	\begin{table}[!t]
		\centering
		\footnotesize
		\renewcommand{\arraystretch}{0.9}
		\setlength{\tabcolsep}{4pt}
		\caption{Reported lower-bound calculations. Component entries are rounded to four decimals; the bounds in the last row are computed from the unrounded values.}\label{tab:numerics}
		\begin{tabular}{lrrrr}
			\toprule
			Deterministic term & Reported bound & $k$ & $M_k$ & $\ell_k$ \\
			\midrule
            $5/8$ & $0.6250$         & 3  & $0.050$ & $0.01147891282$ \\
            $C_3$ & $0.01113528859$  & 4  & $0.020$ & $0.006317757461$ \\
            $C_4$ & $0.005040573276$ & 5  & $0.025$ & $0.003819200575$ \\
            $C_5$ & $0.001015487669$ & 6  & $0.030$ & $0.002518774817$ \\
            && 7  & $0.040$ & $0.001721740355$ \\
            && 8  & $0.040$ & $0.001218625047$ \\
            && 9  & $0.040$ & $0.0009144452461$ \\
            && 10 & $0.040$ & $0.0006807778828$ \\
			\midrule
			Lower bound for $\beta$ & $0.6421913495$ &&& $0.6536702342$ \\
			\bottomrule
		\end{tabular}
	\end{table}

	\section{Derivation of the Upper Bound}\label{sec:upper}
	
	For an upper bound, it suffices to construct a feasible tour and bound its length. We consider a zigzag-style construction, which divides the square region into horizontal strips and connects the points in each strip, similar to \cite{CarlssonYu2025}.

	\subsection{Tour construction}\label{subsec:upper-strip}
	
	Fix a strip height $H>0$ and consider the unit-intensity Poisson process in $[0,\infty)\times[0,H]$. Projecting the points onto the horizontal axis yields a one-dimensional Poisson process with rate $H$. Hence, after ordering the points by their horizontal coordinates, the distances between consecutive horizontal coordinates are independent $\Exp(H)$ random variables. The heights of the points are independent $\Unif[0,H]$ variables, independent of the horizontal distances.

\begin{figure}[t]
\centering
\FIGURE
{\scalebox{1.1}{\begin{tikzpicture}[
  x=1.2cm, y=1.2cm,
  font=\footnotesize,
  tour/.style={
    draw=tourblue, line width=1.1pt,
    line cap=round, line join=round
  },
  pt/.style={
    circle, fill=black, inner sep=0pt, minimum size=3.2pt
  },
  endpt/.style={
    rectangle, draw=black, line width=0.6pt,
    inner sep=0pt, minimum size=7pt
  },
  midarrow/.style={
  postaction={decorate},
  decoration={
    markings,
    mark=at position 0.58 with {
      \fill[tourblue, draw=none]
        (1.8pt,0pt) -- (-1.8pt,1.35pt)
        -- (-1.8pt,-1.35pt) -- cycle;
    }
  }
},
  guide/.style={
    draw=black!35, line width=0.4pt,
    dash pattern=on 1pt off 1.2pt
  },
  dim/.style={draw=black!55, line width=0.45pt, <->}
]
  \coordinate (Q0)  at (0,0.4);
  \coordinate (Q1)  at (0.5,1.7);
  \coordinate (Q2)  at (0.65,0.3);
  \coordinate (Q3)  at (1.2,1.6);
  \coordinate (Q4)  at (1.55,0.9);
  \coordinate (Q5)  at (2.1,1.2);
  \coordinate (Q6)  at (2.6,1.55);
  \coordinate (Q7)  at (3.05,1.05);
  \coordinate (Q8)  at (3.55,0.6);
  \coordinate (Q9)  at (4.0,0.95);
  \coordinate (Q10) at (4.5,0.5);
  \foreach \yy in {0,2} {
    \draw[black!50, line width=0.5pt]
      (-0.45,\yy) -- (4.95,\yy);
    \draw[black!50, line width=0.5pt,
      dash pattern=on 1pt off 1.5pt]
      (-0.80,\yy) -- (-0.45,\yy);
    \draw[black!50, line width=0.5pt,
      dash pattern=on 1pt off 1.5pt]
      (4.95,\yy) -- (5.30,\yy);
  }
  \draw[dim] (-0.25,0.03) -- (-0.25,1.97)
    node[midway, left=1pt] {$H$};
  \draw[guide] (1.55,0) -- (1.55,-0.36);
  \draw[guide] (2.1,1.12) -- (2.1,-0.36);
  \draw[dim] (1.57,-0.28) -- (2.08,-0.28)
    node[midway, below=1pt] {$\Delta_5$};
  \draw[dim] (1.55,0.03) -- (1.55,0.84)
    node[midway, right=0.5pt] {$Y_4$};
\draw[tour, midarrow] (Q0) -- (Q2);
\draw[tour, midarrow] (Q2) -- (Q1);
\draw[tour] (Q1) -- (Q3);
\draw[tour] (Q3) -- (Q4);
\draw[tour, midarrow] (Q4) -- (Q5);
\draw[tour] (Q5) -- (Q6);
\draw[tour] (Q6) -- (Q7);
\draw[tour] (Q7) -- (Q8);
\draw[tour, midarrow] (Q8) -- (Q9);
\draw[tour] (Q9) -- (Q10);
  \foreach \q in {Q0,Q1,Q2,Q3,Q4,Q5,Q6,Q7,Q8,Q9,Q10}
    \node[pt] at (\q) {};
  \foreach \q in {Q0,Q5,Q10}
    \node[endpt] at (\q) {};
  \node[anchor=north] at ($(Q0)+(0,-0.04)$) {$P_0$};
  \node[anchor=east] at ($(Q1)+(0.07,0.02)$) {$P_1$};
  \node[anchor=west] at ($(Q2)+(-0.06,-0.08)$) {$P_2$};
  \node[anchor=west] at ($(Q3)+(-0.02,-0.02)$) {$P_3$};
  \node[anchor=north east] at ($(Q4)+(0.06,0.06)$) {$P_4$};
  \node[anchor=south east] at ($(Q5)+(0.25,0.12)$) {$P_5$};
  \node[anchor=north] at ($(Q10)+(0,-0.09)$) {$P_{10}$};
  \draw[black!50, line width=0.5pt]
    (0,2.12) -- (0,2.26) -- (2.1,2.26) -- (2.1,2.12);
  \draw[black!50, line width=0.5pt]
    (2.1,2.26) -- (4.5,2.26) -- (4.5,2.12);
  \node[anchor=south] at (1.05,2.26) {block 1};
  \node[anchor=south] at (3.30,2.26) {block 2};
  \draw[tour, midarrow] (-0.70,-1.49) -- (-0.30,-1.49)
    node[right=2pt, text=black] {visiting order};
  \node[endpt] at (2.45,-1.49) {};
  \node[pt] at (2.45,-1.49) {};
  \node[right=2pt] at (2.55,-1.49) {shared endpoint};
\end{tikzpicture}}}
{An illustration of the upper bound construction.
\label{fig:upper-construction}}
{Two consecutive blocks in a strip. Each block visits five
new points and ends at the starting point of the next block,
here $P_5$. Arrows show the visiting order selected by the
routing rule. $\Delta_5$ and $Y_4$ illustrate a horizontal
distance and height.}
\end{figure}

	Consider six consecutive points according to their horizontal coordinate values $P_i$ for $i=0,\ldots,5$ (we refer to them as a block). Assume that the leftmost horizontal coordinate is zero, i.e., $P_0=(0,Y_0)$, where $Y_0\sim\Unif[0,H]$. Let their horizontal distances be $\Delta_1,\ldots,\Delta_5$, which implies that
	\begin{align}\label{eq:block-points}
		P_i=(\Delta_1+\cdots+\Delta_i,Y_i),\quad i=1,\ldots,5.
	\end{align}
	Now we consider a path that starts at $P_0$, visits $P_1,P_2,P_3,P_4$ in one of their possible orders, and ends at $P_5$. For a permutation $\pi$ of $\{1,2,3,4\}$, define
	\begin{align*}
		L_\pi=\|P_0-P_{\pi(1)}\|_2+\|P_{\pi(1)}-P_{\pi(2)}\|_2 +\|P_{\pi(2)}-P_{\pi(3)}\|_2+\|P_{\pi(3)}-P_{\pi(4)}\|_2+\|P_{\pi(4)}-P_5\|_2.
	\end{align*}
	A \emph{routing rule} $\mathcal A$ selects one of the $24$ visiting orders as a function of the block's five horizontal distances and six vertical coordinates. We use $C_H^{\mathcal A}$ to denote the resulting path length. The endpoint $P_5$ is then used as the starting point for the next five points, so consecutive paths share a common endpoint and require no additional connecting edge. Figure~\ref{fig:upper-construction} provides an illustration for this strip construction. Then we are able to derive the following results, which establish an upper bound for $\beta$ in terms of $C_H^\cA$. It is worth noting that $C_H^\cA$ has a finite expectation because every such route has length at most $5(\Delta_1+\cdots+\Delta_5)+5H$.
	
	\begin{theorem}\label{thm:strip-upper}
		For every routing rule $\mathcal A$, we have $$\beta\leq   \frac{\E [C_H^{\mathcal A}]}{5}.$$
	\end{theorem}
	
	\proof{Proof.}
	We divide $Q_{mH}=[0,mH]^2$ into the $m$ horizontal strips $[0,mH]\times[(j-1)H,jH]$ ($j=1,\ldots,m$). We first construct a path in each strip from left to right and later reverse alternate paths. We extend the Poisson process beyond $Q_{mH}$. For each strip $j$, we introduce an auxiliary point $P_0^j=(0,(j-1)H+Y_0^j)$, where $Y_0^j\sim\Unif[0,H]$ is independent of all Poisson points. Label the actual Poisson points in this strip in increasing horizontal order as $P_1^j,P_2^j,\ldots,P_{K_m^j}^j,\ldots$, where $K_m^j$ is the number of Poisson points in the strip.
	
	For $t\geq1$, block $t$ starts at $P_{5(t-1)}^j$ and visits the next five points. Let $S_t^j$ be the horizontal coordinate of its starting point, with $S_1^j=0$, and let $C_t^j$ be its path length under $\mathcal A$. Since the distances between consecutive horizontal coordinates and their heights form two independent i.i.d. sequences, the path length $C_t^j$ of each block has the same distribution as $C_H^{\mathcal A}$. Then the number of blocks with $S_t^j<mH$ is $N_m^j=1+\lfloor K_m^j/5\rfloor,$ 
    which implies that 
    $$\E[N_m^j] =\E\left[1+\Big\lfloor \frac{K_m^j}{5}\Big\rfloor\right] \leq 1+\frac{mH^2}{5}.$$
	
	Moreover, by definition, we note that the event $\{S_t^j<mH\}$ and $C_t^j$ are independent. This is because $S_t^j$ depends only on the horizontal distances preceding block $t$ of strip $j$, whereas $C_t^j$ depends on the next five horizontal distances and the heights of the points in block $t$ of strip $j$, all of which are independent of the preceding distances. By Tonelli's theorem, we have
	\begin{align}\label{eq:block-sum}
		\E\left[\sum_{t:S_t^j<mH}C_t^j\right] =\sum_{t\geq1}\Pr(S_t^j<mH)\E[C_t^j]=\E[N_m^j]\E[C_H^{\mathcal A}]\leq\left(1+\frac{mH^2}{5}\right)\E[C_H^{\mathcal A}].
	\end{align}
	The last block may reach beyond $x=mH$. By the memoryless property of the exponential distributions, the distance from this boundary to the next arrival has mean $1/H$, and at most four further horizontal distances are needed to finish that block. Hence the horizontal coordinate of the endpoint of the last block exceeds $mH$ by at most $5/H$ in expectation. Add straight lines that connect $(0,(j-1)H)$ to $P_0^j$ and connect the last block endpoint to $(mH,(j-1)H)$. Hence, these lengths have expected length at most $2H+5/H$. As a result, by the above construction, we obtain a path through all Poisson points with an expected length at most
	\[
	\left(1+\frac{mH^2}{5}\right)\E[C_H^{\mathcal A}]+2H+\frac5H.
	\]
	This is because removing the auxiliary point and all Poisson points beyond $x=mH$ can further shorten the length. Reverse the path in the adjacent strip, for which the two endpoints can be connected at cost no larger than $2H$. Finally, the additional distance required to form a closed tour is at most $\sqrt2mH$. The resulting tour $\mathcal H_m$ through the points in $Q_{mH}$ therefore satisfies
	\[
	\E[\ell(\mathcal H_m)]\leq m\left[\left(1+\frac{mH^2}{5}\right)\E[C_H^{\mathcal A}]+4H+\frac5H\right]+\sqrt2mH.
	\]
	Letting $m$ go to infinity and applying Equation \eqref{eq:poisson}, we have
	\begin{align*}
		\beta &= \lim_{m\to\infty}\frac{\E[\TSP(\cP\cap Q_{mH})]}{(mH)^2}\leq \limsup_{m\to\infty}\frac{\E[\ell(\mathcal H_m)]}{(mH)^2}\\
		&\leq \lim_{m\to\infty} \frac{1}{m^2H^2}\left(m\left(\left(1+\frac{mH^2}{5}\right)\E [C_H^{\mathcal A}] + 4H+\frac{5}{H}\right) + \sqrt{2}mH\right) = \frac15\E [C_H^{\mathcal A}]. \Halmos
	\end{align*}
	\endproof
	In particular, by choosing the shortest path, we have
	\begin{align}\label{eq:ZH}
		\beta\leq \E \left[\frac15\min_\pi L_\pi\right].
	\end{align}

	\subsection{The upper bound}\label{subsec:upper-det}
	
	Theorem~\ref{thm:strip-upper} reduces the problem to calculating $\E[C_H^{\mathcal A}]$, which can be formulated as an eleven-dimensional integral over the five horizontal distances and six vertical coordinates (i.e., $\Delta_1,\ldots,\Delta_5,Y_0,\ldots,Y_5$). Now we attempt to establish a deterministic upper bound for $\E[C_H^{\mathcal A}]$. To this end, we partition the possible locations of the Poisson points into finitely many small subregions and apply the numerical method. Choose integers $N_g,N_y\geq2$. For the horizontal distance between two Poisson points $\Delta\sim\Exp(H)$, define
	\begin{align}\label{eq:gap-bins}
		q_a=\frac1H\log\frac{N_g}{N_g-a}\quad(0\leq a<N_g),\quad I_a=[q_a,q_{a+1})\quad(0\leq a<N_g).
	\end{align}
	Set $q_{N_g}=\infty$. Then the probability that $\Delta\in I_a$ is $1/N_g$. For the vertical values of each Poisson point $Y\sim \Unif[0,H]$, let $J_b=[bH/N_y,(b+1)H/N_y)$ for $b=0,\ldots,N_y-1$, with $H$ included in the last interval. Then the probability that $Y\in J_b$ is $1/N_y$. Since the five horizontal distances and six vertical coordinates are independent, the probability that $(\Delta_1,\ldots,\Delta_5,Y_0,\ldots,Y_5)$ lies in a particular combination of intervals $\sigma = (I_{a_1},\ldots,I_{a_5},J_{b_0},\ldots,J_{b_5})$ is $(N_g^5N_y^6)^{-1}$. 
	
	Now we assume that $(\Delta_1,\ldots,\Delta_5,Y_0,\ldots,Y_5)\in \sigma$ holds and attempt to bound the expected length of the path from above. To this end, we compute the conditional moments of the horizontal distances. For $a=0,\ldots,N_g-1$, let 
    \[m_a=\E[\Delta\mid\Delta\in I_a], \quad s_a=\E[\Delta^2\mid\Delta\in I_a].\] 
    Define $\Phi_H(x)=x^2+2x/H+2/H^2$ for convenience. Then we obtain the following equation from direct computation:
	\begin{align}\label{eq:gap-moments}
		m_a&=(N_g-a)q_a-(N_g-a-1)q_{a+1}+\frac1H, \\
        s_a&=(N_g-a)\Phi_H(q_a)-(N_g-a-1)\Phi_H(q_{a+1}),
	\end{align}
	for $a<N_g-1$. Moreover, we have 
    $$m_{N_g-1}=q_{N_g-1}+\frac1H, s_{N_g-1}=\Phi_H(q_{N_g-1}).$$

	We next derive an upper bound on the conditional expected length of each possible edge $(P_r,P_s)$ for $0\leq r<s\leq5$. By Equation \eqref{eq:gap-moments}, we establish that
	\begin{align}\label{eq:horizontal-second-moment}
		\E\left[\left(\sum_{i=r+1}^{s}\Delta_i\right)^2\middle|\sigma\right]&=\sum_{i=r+1}^{s}\E[\Delta_i^2\mid\sigma]+2\sum_{r+1\leq i<j\leq s}\E[\Delta_i\Delta_j\mid\sigma]\nonumber\\
		&=\sum_{i=r+1}^{s}s_{a_i}+2\sum_{r+1\leq i<j\leq s}m_{a_i}m_{a_j},
	\end{align}
	where the last equation holds because $\Delta_i$ and $\Delta_j$ are independent. Moreover, conditional on $\sigma$, each $Y_j$ is uniformly distributed on $J_{b_j}=[b_jH/N_y,(b_j+1)H/N_y)$. Therefore, we have
	\begin{align*}
		\E[Y_j\mid\sigma]&=\left(b_j+\frac12\right)\frac{H}{N_y},\quad \Var(Y_j\mid\sigma)=\frac{H^2}{12N_y^2}.
	\end{align*}
	Since $Y_r$ and $Y_s$ are independent conditional on $\sigma$, we have
	\begin{align}\label{eq:vertical-second-moment}
		\E[(Y_s-Y_r)^2\mid\sigma] &=\Var(Y_s\mid\sigma)+\Var(Y_r\mid\sigma) +\left(\E[Y_s\mid\sigma]-\E[Y_r\mid\sigma]\right)^2\nonumber\\
        &=\left(\frac{H}{N_y}\right)^2 \left((b_s-b_r)^2+\frac16\right).
	\end{align}
	Combining \eqref{eq:horizontal-second-moment} and \eqref{eq:vertical-second-moment}, we have
	\begin{align}\label{eq:edge-second-moment}
		M_{rs}(\sigma) &:=\E[\|P_s-P_r\|_2^2\mid\sigma] = \E\left[\left(\sum_{i=r+1}^{s}\Delta_i\right)^2+(Y_s-Y_r)^2\Big|\sigma\right] \nonumber \\
		&=\sum_{i=r+1}^{s}s_{a_i} +2\sum_{r+1\leq i<j\leq s}m_{a_i}m_{a_j} +\left(\frac{H}{N_y}\right)^2 \left((b_s-b_r)^2+\frac16\right).
	\end{align}
	For $r>s$, define $M_{rs}(\sigma)=M_{sr}(\sigma)$ by symmetry. For a permutation $\pi$ of $\{1,2,3,4\}$, define
	\begin{align*}
		B_\pi(\sigma) :=\sqrt{M_{0,\pi(1)}(\sigma)} +\sqrt{M_{\pi(1),\pi(2)}(\sigma)} +\sqrt{M_{\pi(2),\pi(3)}(\sigma)} +\sqrt{M_{\pi(3),\pi(4)}(\sigma)}+\sqrt{M_{\pi(4),5}(\sigma)}.
	\end{align*}
	By Jensen's Inequality, we note that 
    $$M_{rs}(\sigma)\geq (\E[\|P_s-P_r\|_2\mid\sigma])^2,$$ 
    which implies that $\E[L_\pi\mid\sigma]\leq B_\pi(\sigma)$. For each cell $\sigma$, define the routing rule $\mathcal A$ as selecting the visiting order $\pi_\sigma$ that minimizes the value of $B_\pi(\sigma)$ among all permutations. This is a routing rule in the sense of Section~\ref{subsec:upper-strip}, and thus Theorem~\ref{thm:strip-upper} applies. Hence, we have 
    \[
        \E[C_H^{\mathcal A}\mid\sigma] =\E[L_{\pi_\sigma}\mid\sigma] \leq B_{\pi_\sigma}(\sigma) =\min_\pi B_\pi(\sigma).
    \]
    Note that each of the $N_g^5N_y^6$ cells has probability $(N_g^5N_y^6)^{-1}$. Applying Theorem \ref{thm:strip-upper} and the above inequality over all cells gives
	\begin{align}\label{eq:finite-cell-upper}
		\beta\leq \frac{1}{5}\E[C_H^{\mathcal A}] \leq\frac{1}{5N_g^5N_y^6}\sum_{\sigma}\min_{\pi}B_\pi(\sigma).
	\end{align}
	We exploit several exact symmetries to reduce the number of cells that must be evaluated. First, a vertical translation to all points does not change any pairwise distance. Second, vertical reflection preserves all pairwise vertical differences. Finally, reversing both the horizontal distances and the vertical coordinate indices also preserves the value of $\min_\pi B_\pi(\sigma)$. For each equivalence class, we only need to evaluate one representative cell, which can reduce the computational complexity.
	
	In our numerical experiments, we take $H=2$, $N_g=12$, and $N_y=20$. Since $q_a$ is irrational, we enclose it by rational numbers. Because the coefficients of $q_a$ and $q_{a+1}$ in \eqref{eq:gap-moments} have opposite signs, $\Phi_H$ is increasing on $[0,\infty)$, and every $m_a$ is positive, replacing $q_a$ by a rational upper bound and $q_{a+1}$ by a rational lower bound yields upper bounds on $m_a$ and $s_a$, and hence on $M_{rs}(\sigma)$ and $B_\pi(\sigma)$. We replace each square root by a rational upper bound with denominator $10^{12}$. By \eqref{eq:finite-cell-upper}, we have
	\[
	\beta\leq \frac{70149051060530822644063354/10^{12}}{5\cdot12^5\cdot20^6} = \frac{70149051060530822644063354}{5\cdot10^{12}\cdot12^5\cdot20^6}<0.8810.
	\]

	\subsection{High-probability upper bound}\label{subsec:upper-hp}
	
	In this section, we establish a high-probability upper bound for $\beta$. Let $Z_H=\min_\pi L_\pi/5$. By Equation \eqref{eq:ZH}, we have $\beta\leq\E[Z_H]$. Since $Z_H$ is unbounded, we fix $B>H$ and consider the identity
    \[ \E[Z_H] =\E[\min\{Z_H,B\}]+\E[(Z_H-B)_+]. \]
    Therefore, to obtain an upper bound for $\E[Z_H]$, it suffices to consider $\E[\min\{Z_H,B\}]$ and $\E[(Z_H-B)_+]$. To begin with, we have the following result, which bound $\E[(Z_H-B)_+]$ from above.
	
	\begin{lemma}\label{lem:upper-tail}
		For every $B>H$, define $x_B=5H(B-H)$. Then
		\begin{align}\label{eq:upper-tail}
			\E[(Z_H-B)_+]\leq \frac{e^{-x_B}}{5H}\left(5+4x_B+\frac{3x_B^2}{2}+\frac{x_B^3}{3}+\frac{x_B^4}{24}\right).
		\end{align}
	\end{lemma}
	
	\proof{Proof.}
	Consider the path that visits the points from left to right: $P_0\to P_1\to P_2\to P_3\to P_4\to P_5$. For its $i$th edge, the Triangle Inequality implies that $\|P_i-P_{i-1}\|_2 \leq\Delta_i+|Y_i-Y_{i-1}|$. The five vertical differences are each at most $H$, so $Z_H\leq H+(\Delta_1+\cdots+\Delta_5)/5$. Let $S=\Delta_1+\cdots+\Delta_5$. Then $S$ has the Erlang distribution with parameters $(5,H)$, and
	\begin{align*}
		\Pr(S>s)=e^{-Hs}\left(1+Hs+\frac{(Hs)^2}{2}+\frac{(Hs)^3}{6}+\frac{(Hs)^4}{24}\right).
	\end{align*}
	Thus, we obtain that
	\begin{align*}
		\E[(Z_H-B)_+] &\leq \E\left[\left(H+\frac{S}{5}-B\right)_+\right] = \int_B^\infty\Pr\left(H+\frac{S}{5}>s\right)\d s  = \frac{1}{5}\int_{5(B-H)}^\infty\Pr\left(S>t\right)\d t\\
		&=\frac{1}{5}\int_{5(B-H)}^\infty\left(e^{-Ht}\left(1+Ht+\frac{(Ht)^2}{2}+\frac{(Ht)^3}{6}+\frac{(Ht)^4}{24}\right)\right)\d t\\
		&=\frac{e^{-x_B}}{5H}\left(5+4x_B+\frac{3x_B^2}{2}+\frac{x_B^3}{3}+\frac{x_B^4}{24}\right),
	\end{align*}
	which completes the proof.
	\Halmos
	\endproof

	With Lemma \ref{lem:upper-tail}, we are now ready to bound $\E[Z_H]$. Fix $H>0$, an integer $n\geq2$, $0<\delta<1$, and a nonempty finite set $\mathcal B\subset(H,\infty)$ before sampling, and let $Z_{H,1},\ldots,Z_{H,n}$ be independent copies of $Z_H$. For $B\in\mathcal B$, let $\overline X_B$ and $s_B^2$ be the sample mean and the sample variance of the clipped values $\min\{Z_{H,1},B\},\ldots,\min\{Z_{H,n},B\}$, that is,
	\[
	\overline X_B=\frac1n\sum_{i=1}^n\min\{Z_{H,i},B\},\qquad s_B^2=\frac1{n-1}\sum_{i=1}^n\bigl(\min\{Z_{H,i},B\}-\overline X_B\bigr)^2.
	\]
	Then we have the following property.
	
	\begin{proposition}\label{prop:upper-hp}
		With probability at least $1-\delta$, we have
		\begin{align}\label{eq:Udelta}
			\beta\leq U_\delta:=\min_{B\in\mathcal B}\left\{ \overline X_B +\sqrt{\frac{2s_B^2\log(2|\mathcal B|/\delta)}{n}}+\frac{7B\log(2|\mathcal B|/\delta)}{3(n-1)}+\frac{e^{-x_B}}{5H}\left(5+4x_B+\frac{3x_B^2}{2}+\frac{x_B^3}{3}+\frac{x_B^4}{24}\right)\right\},
		\end{align}
		where $|\mathcal B|$ is the number of elements of $\mathcal B$.
	\end{proposition}
	
	\proof{Proof.}
	Note that the clipped variables $\min\{Z_{H,i},B\}$ are independent and lie in the closed interval $[0,B]$, and that $\E[Z_H]=\E[\min\{Z_H,B\}]+\E[(Z_H-B)_+]$. Applying the Empirical Bernstein Inequality \citep{MaurerPontil2009} to the variables $\min\{Z_{H,i},B\}/B$, we find that the following inequality holds with probability at least $1-\delta/|\mathcal B|$:
	\begin{align}\label{4}
		\E[\min\{Z_H,B\}]\leq\overline X_B +\sqrt{\frac{2s_B^2\log(2|\mathcal B|/\delta)}{n}}+\frac{7B\log(2|\mathcal B|/\delta)}{3(n-1)}.
	\end{align}
	Combining Lemma~\ref{lem:upper-tail} with the above inequality, we conclude that the following inequality holds with probability at least $1-\delta/|\mathcal B|$:
	\begin{align*}
		\E[Z_H] \leq \overline X_B +\sqrt{\frac{2s_B^2\log(2|\mathcal B|/\delta)}{n}}+\frac{7B\log(2|\mathcal B|/\delta)}{3(n-1)} +\frac{e^{-x_B}}{5H}\left(5+4x_B+\frac{3x_B^2}{2}+\frac{x_B^3}{3}+\frac{x_B^4}{24}\right).
	\end{align*}
	Applying the union bound over all $B\in\mathcal B$, the above inequalities hold with probability at least $1-\delta$. Hence, on this event, we have $\beta\leq\E[Z_H]\leq U_\delta$, which completes the proof.
	\Halmos
	\endproof
	
	With $H=2$, $n=5\times 10^6$, $\delta=10^{-4}$, and $\mathcal B=\{3.00,3.25,3.50,\ldots,5.00\}$, the minimum value is taken when $B=4.25$, for which $x_B=5H(B-H)=22.5$. At this level, the clipped sample mean is $\overline X_B\approx 0.8742562815$, the sample variance is $s_B^2\approx0.0585344834$. Since $|\mathcal B|=9$, we have $\log(2|\mathcal B|/\delta)=\log(180000)$. The sum of the two concentration terms in \eqref{eq:Udelta} is approximately $0.0005562812301$, and the tail term of Lemma~\ref{lem:upper-tail} is approximately $2.5936729064\times10^{-7}$. Therefore, \eqref{eq:Udelta} gives 
    \[U_\delta\approx 0.8742562815 + 0.0005562812301 +2.5936729064\times10^{-7}\approx0.8748128221<0.8749.\] 
    By Proposition \ref{prop:upper-hp}, we have $\Pr(\beta\leq 0.8749)\geq1-10^{-4}$.

	\section{Conclusion}\label{sec:conclusion}

    We have improved the best known bounds on the Euclidean TSP constant from $0.6277\leq\beta\leq0.90367$ to $0.6421\leq\beta\leq0.8810$, and shown that $0.6536\leq\beta\leq0.8749$ holds with probability at least $1-2\times10^{-4}$ with respect to the sampling procedure. The lower bound follows from a new inequality relating tour length to path covers of threshold graphs; the upper bound follows from a strip construction with optimally ordered blocks of five points. Both bounds are evaluated by computer-assisted methods.
	
	Further improvements may come from finer partitions, sharper disk-union area bounds, and larger blocks on the upper side. Closing the remaining distance to the empirical value $\beta\approx0.7124$, however, may require new techniques.

	\section*{Disclosure of AI Use}
	
	The authors used the AI tool GPT-5.6 Sol Pro in preparing this paper. The authors have reviewed and verified all results, and take full responsibility for the content of the paper.

	\section*{Code and Data}

	The Python code and instructions for reproducing the numerical results in this paper are available upon request. No external datasets are required.

	\bibliography{reference}
	\bibliographystyle{informs2014}

    \newpage
    \ECSwitch

\begin{APPENDICES}

\section{Main notation}\label{app_a}
\begin{itemize}[leftmargin=*,labelindent=0pt]
    \item $n$: number of locations.
    \item $X_{1},\ldots,X_{n}$: independent uniform random points in the unit square.
    \item $\TSP(X_{1},\ldots,X_{n})$: length of a shortest Euclidean tour through $X_{1},\ldots,X_{n}$.
    \item $\beta$: the TSP constant.
    \item $\cP$: a Poisson process with unit intensity on the entire plane.
    \item $\cP_n$: a Poisson process with intensity $n$ on the unit square $[0,1]^2$.
    \item $Q_t$: square $[0,t]^2$.
    \item $p(G)$: the minimum number of vertex-disjoint paths covering the vertices of the finite graph $G$.
    \item $N_H(r)$: the number of edges of length greater than $r$ in the Hamiltonian cycle $H$.
    \item $\cC(G)$: the collection of connected components of the graph $G$.
    \item $\operatorname{cc}(G)$: the number of connected components of the graph $G$.
    \item $\tau(G)$: the number of spanning trees of graph $G$.
    \item $G_r(X)$: the graph with vertex set $X$, where the edge set $E$ is:
	\begin{align*}
		E=\left\{\{x_1,x_2\}:x_1,x_2\in X,\ x_1\neq x_2,\ \|x_1-x_2\|_2\leq r\right\}.
	\end{align*}
    \item $\boldsymbol y$: a brevity notation of the collection $y_1,\ldots,y_k$.
    \item $A(\boldsymbol{y})$: the area of the union of unit balls centered at the points of $\boldsymbol{y}$.
    \item $\boldsymbol Y=(Y_1,\ldots,Y_k)$: planar points generated by the tree sampling procedure, with $Y_1=(0,0)$.
    \item $f_k$: density function of the sampled points $Y_2,\ldots,Y_k$.
    \item $B(x,r)$: disk centered at $x$ with radius $r$.
    \item $\Gamma$: the Gamma function.
    \item $X_t$: the Poisson points that lie in $Q_t$, i.e., $X_t = \cP\cap Q_t$.
    \item $T_t$: the length of the shortest TSP tour among $X_t$.
    \item $\|y\|_2$: the $L^2$ norm of $y$.
    \item $C_k$: the integral in the lower bound for $\beta$; see \eqref{eq:Ck}.
    \item $W_k$: an unbiased estimator of $C_k$.
    \item $\ell_k$: lower confidence bound for $C_k$.
    \item $n_k$: number of samples for estimating $C_k$.
    \item $M_k$: the clipping level for estimating $C_k$.
    \item $H$: the strip height.
    \item $\pi$: a permutation of $\{1,2,3,4\}$.
    \item $\cA$: a routing rule that selects one of the $24$ orders as a function of the block's point locations $P_0$ thru $P_5$. Here, $P_1,\ldots,P_4$ are visited between $P_0$ and $P_5$.
    \item $\Delta_i$: horizontal gaps between consecutive block points $P_i$.
    \item $Y_i$: vertical coordinates of the block points $P_i$.
    \item $C_H^\cA$: length of the block path selected by $\mathcal A$.
    \item $Z_H$: shortest block path length divided by five.
    \item $B_\pi$: an upper bound for $\E[L_\pi\mid\sigma]$.
    \item $U_\delta$: upper confidence bound for $\beta$ with failure probability at most $\delta$.
\end{itemize}

\end{APPENDICES}
	
\end{document}

%% file: premeable.tex
\usepackage[utf8]{inputenc}
\usepackage[T1]{fontenc}
\usepackage[english]{babel}
\usepackage{microtype}
\usepackage{mathrsfs,mathtools,accents}
\usepackage{xifthen,enumitem,comment}
\setlist[enumerate,1]{label=\normalfont{(\Roman*)},leftmargin=*}
\usepackage{xcolor}
\usepackage{etoolbox}
\makeatletter
\patchcmd{\env@cases}{1.2}{0.96}{}{}
\makeatother
\usepackage{algorithm,algorithmicx,algpseudocode}
\usepackage[hypertexnames=false]{hyperref}
\hypersetup{%
  colorlinks=true,
  linkcolor={blue!66!black},
  linkbordercolor=red,
  citecolor={blue!66!black},
}
\usepackage{tikz,pgfplots,subfigure}
\usetikzlibrary{matrix,decorations.pathreplacing}

\ifx\argmax\undefined\DeclareMathOperator*{\argmax}{arg\,max}\fi
\ifx\argmin\undefined\DeclareMathOperator*{\argmin}{arg\,min}\fi

\renewcommand{\Pr}{\mathbb{P}}
\providecommand{\E}{\mathbb{E}}
\providecommand{\R}{\mathbb{R}}

\providecommand{\Var}{\mathbb{V}\textrm{ar}}

\providecommand{\cA}{\mathcal{A}}

\providecommand{\cC}{\mathcal{C}}

\providecommand{\cP}{\mathcal{P}}
\providecommand{\cQ}{\mathcal{Q}}